\documentclass{siamltex}
\usepackage{amsmath,amssymb}
\usepackage{latexsym}
\usepackage{graphicx,color}
\renewcommand{\theequation}{\thesection.\arabic{equation}}

\newtheorem{thm}{Theorem}[section]
\newtheorem{cor}[thm]{Corollary}

\newtheorem{rmk}[thm]{Remark}

\newcommand{\norm}[1]{\left\Vert#1\right\Vert}

\newcommand{\eps}{\varepsilon}

\renewcommand{\figurename}{\scriptsize Figure}
\renewcommand\appendix{\par
  \setcounter{section}{0}
  \setcounter{subsection}{0}
  \setcounter{figure}{0}
  \setcounter{table}{0}
  \renewcommand\thesection{Appendix \Alph{section}}
  \renewcommand\theequation{\Alph{section}.\arabic{equation}}
  \renewcommand\thefigure{\Alph{section}.\arabic{figure}}
  \renewcommand\thetable{\Alph{section}.\arabic{table}}
  \renewcommand\thethm{\Alph{section}.\arabic{thm}}
}

\newcommand{\E}{\mathbb{E}}

\newcommand{\cT}{\mathcal{T}}

\newcommand{\veps}{\varepsilon}
\newcommand{\Ome}{\Omega}
\newcommand{\ome}{\omega}

\newcommand{\nab}{\nabla}

\def\be{\begin{equation}}
\def\ee{\end{equation}}
\def\br{\begin{eqnarray}}
\def\er{\end{eqnarray}}

\numberwithin{equation}{section}

\begin{document}
\title{A multi-modes Monte Carlo finite element method for elliptic partial differential 
equations with random coefficients}
\markboth{X. FENG, J. LIN AND C. LORTON}{A MULTI-MODES METHOD FOR SPDEs}

\author{
Xiaobing Feng \thanks{Department of Mathematics, The University of Tennessee,
Knoxville, TN 37996, U.S.A. ({\tt xfeng@math.utk.edu}). }
\and
Junshan Lin\thanks{Department of Mathematics and Statistics, Auburn University,
Auburn, AL 36849, U.S.A. ({\tt jzl0097@auburn.edu}). }
\and
Cody Lorton\thanks{ Department of Mathematics and Statistics, University of West Florida,
Pensacola, FL 32514, U.S.A.({\tt clorton@uwf.edu}). }
}


\maketitle

\begin{abstract}
This paper develops and analyzes an efficient numerical method for solving elliptic partial 
differential equations, where the diffusion coefficients are random perturbations 
of deterministic diffusion coefficients.  The method is based upon a multi-modes 
representation of the solution as a power series of the perturbation parameter, and the
Monte Carlo technique for sampling the probability space.  One key feature of the proposed 
method is that the governing equations for all the expanded mode functions
share the same deterministic diffusion coefficients, thus an efficient direct solver
by repeated use of the $LU$ decomposition matrices can be employed for solving the finite 
element discretized linear systems. It is shown that the computational complexity of 
the whole algorithm is comparable to that of solving a few
deterministic elliptic partial differential equations using the $LU$ director solver.
Error estimates are derived for the method, and numerical experiments are provided to 
test the efficiency of the algorithm and validate the theoretical results.
\end{abstract}



\section{Introduction}\label{sec-1}

There has been increased interest in numerical approximation of
random partial differential equations (PDEs) in recent years, due to the need
to model the uncertainties or noises that arise in industrial and engineering applications
\cite{Babuska_Nobile_Tempone10, Babuska_Tempone_Zouraris04,Caflisch, FGPS, Ishimaru, Ronan_Sarkis}.  
To solve random boundary value problems numerically, the Monte Carlo method
obtains a set of independent identically distributed (i.i.d.) solutions by sampling 
the PDE coefficients, and calculates the mean of the solution via a statistical 
average over all the sampling in the probability space \cite{Caflisch}. 
The stochastic Galerkin method, on the other hand, reduces the SPDE to a high dimensional deterministic 
equation by expanding the random coefficients in the equation using the Karhunen-Lo\`{e}ve or 
Wiener Chaos expansions \cite{Babuska_Nobile_Tempone10,
Babuska_Tempone_Zouraris04, Babuska_Tempone_Zouraris05, DBO, EEU, Liu_Riviere13, Ronan_Sarkis, Xiu_Karniadakis1, Xiu_Karniadakis2}.
In general, these two methods become computationally expensive when a large number of degrees of freedom is involved 
in the spatial discretization, particularly for three dimensional boundary value problems. 
The Monte Carlo method requires solving the boundary value problem 
many times with different sampling coefficients, while the stochastic Galerkin method usually 
leads to a high dimensional deterministic equation that may be too expensive to solve.

Recently, we have developed a new efficient multi-modes Monte Carlo method for modeling acoustic wave 
propagation in weakly random media \cite{FLL}. To solve the governing random Helmholtz equation,
the solution is represented by a sum of mode functions, where each 
mode satisfies a Helmholtz equation with deterministic coefficients and a random source.
The expectation of each mode function is then computed using a Monte Carlo interior penalty 
discontinuous Galerkin (MCIP-DG) method.
We take the advantage that the deterministic Helmholtz operators for all the modes are identical,
and employ an $LU$ solver for obtaining the numerical solutions.
Since the discretized equations for all the modes have the same constant coefficient matrix,
by using the $LU$ decomposition matrices repeatedly, the solutions for all samplings of mode functions
are obtained in an efficient way by performing simple forward
and backward substitutions. This leads to a tremendous saving in the computational costs.
Indeed, as discussed in \cite{FLL}, the computational complexity
of the proposed algorithm is comparable to that of solving a few
deterministic Helmholtz problem using the $LU$ direct solver. 

In this paper, we extend the multi-modes Monte Carlo method for approximating the solution 
to the following random elliptic problem:
\begin{alignat}{2}\label{bvp}
-\nabla \cdot \left(a(\omega,\cdot)\nabla u^\veps (\omega,\cdot)\right) &= f(\omega,\cdot)  &&\qquad \mbox{in } D, \\
u^\veps(\omega,\cdot) &= 0 &&\qquad \mbox{on } \partial D. \label{bnd_cond}
\end{alignat}
Here $D$ is a bounded Lipschitz domain in $\mathbb{R}^d$ ($d=1, 2, 3$),
$a(\omega,x)$ and $f(\omega,x)$ are random fields with continuous and bounded covariance functions.
Let $(\Ome,\mathcal{F}, P)$ be a probability space with sample space $\Ome$,
$\sigma-$algebra $\mathcal{F}$ and probability measure $P$.  We consider the case when 
the diffusion coefficient $a(\omega,x)$ in \eqref{bvp} is a small random perturbation of 
some deterministic diffusion coefficient such that
\begin{align}\label{diff_coeff}
a(\ome, \cdot):=a_0(\cdot)+\veps\eta(\omega,\cdot).
\end{align}
Here $a_0\in W^{1,\infty}(D)$, $\veps$ represents the magnitude of the random fluctuation, and
$\eta\in L^2(\Ome,W^{1,\infty}(D))$ is a random function satisfying
$$
P\left\{\ome\in\Ome;\, ||{\eta(\ome,\cdot)}||_{W^{1,\infty}(D)} \le b_0 \right\}=1
$$
for some positive constants $b_0$. The readers are referred to Section \ref{sec-2} for the 
definition of the function spaces $W^{1,\infty}(D)$ and $L^2(\Ome,W^{1,\infty}(D))$. 
The random diffusion coefficient \eqref{diff_coeff} can be interpreted as diffusion through a random 
perturbation of some deterministic background medium.
It is required that $a(\omega,x)$ is uniformly coercive. 
That is, there exists a positive constant $\underline{a}$ such that 
$$
P\left\{\ome\in\Ome;\, \min_{x\in\bar D}{ a(\omega, x))}>\underline{a} \right\}=1.
$$

The proposed numerical method is based on the following multi-modes expansion of the solution:
 $$u^{\veps}(\omega,x)  = \sum_{n=0}^\infty \veps^n u_n(\omega,x).$$
It is shown in this paper that the expansion series converges to $u^{\veps}$ and each 
mode $u_n$ satisfies an elliptic equation with deterministic coefficients and a random source.
We apply the Monte Carlo method for sampling over the probability space $\Omega$ and 
use the finite element method for solving the boundary value problem for $u_n$ at each realization.
An interesting and important fact of the mode expansion is that 
all $u_n$ share the same deterministic elliptic operator $\nabla\cdot(a_0\nabla)$, hence 
the $LU$ decomposition of the finite element stiff matrix can be used repeatedly.
As such, solving for $u_n(\omega,x)$ for each $n$ and at each realization $\omega=\omega_j$ 
only involve simple forward and backward substitutions with the $L$ and $U$ matrices, 
and the computational complexity for the whole algorithm can be significantly reduced.
It should be pointed out that here the randomly perturbed diffusion coefficient $a(\omega,x)$ 
appears in the leading term of the elliptic differential operator, while for the Helmholtz 
equation considered in \cite{FLL}, the random coefficient only appears in the low order term.
This results in essential differences in both computation and analysis when the multi-modes 
expansion idea is applied to these two problems.  

The rest of the paper is organized as follows.  We begin with introducing some space 
notations in Section \ref{sec-2} and discuss the well-posedness of the problem 
\eqref{bvp}-\eqref{bnd_cond}.  In Section \ref{sec-3}, we introduce the multi-modes 
expansion of the solution as a power series of $\veps$, and derive the error estimation 
for its finite-modes approximation.  The details of the multi-modes Monte Carlo method 
are given in Section \ref{sec-4}, where the computational complexity of the algorithm and 
the error estimations for the numerical solution are also obtained. 
Several numerical examples are provided in Section \ref{sec-5} to demonstrate the efficiency 
of the method and to validate the theoretical results.  We end the paper with a discussion on 
generalization of the proposed numerical method to more general random PDEs in Section \ref{sec-6}.

\section{Preliminaries}\label{sec-2}
Standard space notations will be adopted in this paper \cite{Adams, Gilbarg_Trudinger01, LM}. For example,
$L^2(D)$ denotes the Hilbert space of all square integrable functions equipped with the 
inner product $(f,g)_D:=\displaystyle{\int_\Omega f g\,dx}$ and the induced norm
$$||u||_{L^2(D)} = \left(\int_D |u(x)|^2 dx \right)^{\frac12},$$
and $L^\infty(D)$ is the set of bounded measurable functions equipped with the norm
$$  ||u||_{L^\infty(D)} = \underset{x\in D}{\mbox{esssup}} |u(x)|. $$
For a positive integer $m$ and a fraction $s=m+\sigma$ with some $\sigma\in(0,1)$,
we define the Sobolev spaces $H^{m}(D)$ and $H^{s}(D)$ as
\begin{align}
H^{m}(D) &:=\{ u\in L^2(D);\; ||u||_{H^m(D)} < \infty \},\\
H^{s}(D) &:=\{ u\in L^2(D);\; ||u||_{H^s(D)} < \infty \}, 
\end{align}
where 
\begin{eqnarray*}
||u||_{H^m(D)}^2 &:=& \sum_{|\alpha|\le m} ||\partial^\alpha u||^2_{L^2(D)},\\
||u||_{H^s(D)}^2 &:=&  ||u||_{H^m(D)}^2 + \sum_{|\alpha|=m}\int\int \dfrac{|\partial^\alpha u(x)- \partial^\alpha u(y)|^2}{|x-y|^{n+2\sigma}}dxdy.
 \end{eqnarray*}
We also define $H_0^{m}(D)$ and $H_0^{s}(D)$ to be the subspaces of $H^{m}(D)$ and $H^{s}(D)$ 
with zero trace, and
 $H^{-m}(D)$ and $H^{-s}(D)$ as the dual spaces of  $H^m_0(D)$ and $H^s_0(D)$, respectively.
The Sobolev space $W^{1,\infty}(D)$ is given by
$$W^{1,\infty}(D) := \{  u\in L^{\infty}(D);\;  ||u||_{W^{1,\infty}(D)} < \infty \},$$
where $ ||u||_{W^{1,\infty}(D)} := ||u||_{L^\infty(D)}+||\nabla u||_{L^{\infty}(D)}$.
Finally, for a Banach space $X$, let $L^2(\Ome, X)$ denote the space of all measurable 
function $u: \Omega \to X$ such that $\|u\|_{L^2(\Omega,X)}:=\Bigl(\displaystyle{\int_\Omega 
\|u(\omega,\cdot)\|_{X}^2 d\omega \Bigr)^{\frac12} < \infty} $. Later in this paper, we shall take
$X$ to be $H^m(D)$, $H^s(D)$, or $W^{1,\infty}(D)$.

For a given source function $f\in L^2(\Ome, H^{-1}(D))$, a weak solution for the problem 
\eqref{bvp}--\eqref{bnd_cond} is defined as a function  $u\in L^2(\Ome, H_0^1(D))$ such that
\begin{equation}\label{var_prob}
\int_\Ome  \bigl(a \nab u^\veps, \nab v\bigr)_D \; dP 
= \int_\Ome \langle f,  v \rangle_D \; dP \qquad\forall
v\in L^2(\Ome, H_0^1(D)),
\end{equation}
where $(\cdot,\cdot)_D$ stands for the inner product on $L^2(D)$, and 
$\langle \cdot,\cdot \rangle_D$ denotes the dual product on $H^{-1}(D)\times H_0^1(D)$.
Following the standard energy estimates and applying the Lax-Milgram theorem, 
it can be shown that \eqref{var_prob} attains a unique solution in $u\in L^2(\Ome, H_0^1(D))$ 
\cite{Babuska_Tempone_Zouraris04, Gilbarg_Trudinger01, LM}.  If $f\in L^2(\Ome, H^{-1+\sigma}(D))$ 
with $\sigma\in(0,1]$ and the boundary of the domain $D$ is sufficiently smooth,
then elliptic regularity theory gives rise to the following energy estimate (cf.  \cite{LM})
\begin{equation}\label{u_reg}
\E(\norm{u^\veps}_{H^{1+\sigma}(D)}^2) \leq C \;\E(\norm{f}_{H^{-1+\sigma}(D)}^2),
\end{equation}
where $C$ is some constant deepening on $a(\omega, x)$ and the domain $D$.
In particular, when $\sigma=1$, or equivalently $f\in L^2(\Ome, L^{2}(D))$, we have
\begin{equation}\label{H2_estimate}
\E(\norm{u^\veps}_{H^2(D)}^2) \leq C \;\E(\norm{f}_{L^2(D)}^2).
\end{equation}

\section{Multi-modes expansion of the solution}\label{sec-3}
Our multi-modes Monte Carlo method will be based on the following multi-modes representation 
for the solution of \eqref{bvp}--\eqref{bnd_cond}
\begin{equation}\label{u_exp}
u^{\veps}(\omega,x)  = \sum_{n=0}^\infty \veps^n u_n(\omega,x),
\end{equation}
where the convergence of the series will be justified below.  

Substituting the above expansion into \eqref{bvp} and matching the coefficients of $\veps^n$ 
order terms for $n=0,1,2,\cdots$, it follows that
\begin{eqnarray}\label{un_pde}
-\nabla \cdot (a_0 \nabla u_0(\omega,\cdot)  ) &=& f(\omega,\cdot),  \label{u0_pde} \\
-\nabla \cdot (a_0 \nabla u_n(\omega,\cdot) )  &=& \nabla \cdot (\eta \nabla u_{n-1}(\omega,\cdot) )
\qquad \mbox{for} \; n\geq 1 \label{un1_pde}.
\end{eqnarray}
Correspondingly, the boundary condition for each mode function $u_n$ is given by
\begin{equation}\label{un_bnd_cond}
u_n(\omega,\cdot) =0  \qquad \mbox{on} \; \partial D  \qquad\mbox{for }   n\geq 0.
\end{equation}
It is clear that each mode satisfies an elliptic equation with the same deterministic 
coefficient $a_0$ and a random source term. On the other hand, for $n\ge 1$, the source 
term in the PDE is given by the previous mode $u_{n-1}$.
This implies that the mode $u_n$ has to be solved recursively for $n=0, 1, 2, \cdots$.
We first derive the energy estimate for each mode $u_n$.

\medskip

\begin{thm}\label{un_energy}
There exists a unique solution $u_n \in L^2(\Ome, H_0^1(D))$ to the problem \eqref{u0_pde} and \eqref{un_bnd_cond} for $n=0$,
and the problem \eqref{un1_pde}--\eqref{un_bnd_cond} for $n\geq 1$. In addition, if $f\in L^2(\Ome, H^{-1+\sigma}(D))$ for $\sigma\in(0,1]$, there holds
\begin{equation}\label{un_est}
\E(\norm{u_n}_{H^{1+\sigma}(D)}^2) \leq  C_0^{n+1} \;\E(\norm{f}_{H^{-1+\sigma}(D)}^2)
\end{equation}
for some constant $C_0$ independent of $n$ and $\veps$.
\end{thm}

\medskip

\begin{proof}
For $n=0$, the existence of the weak solutions can be deduced from the Lax-Milgram Theorem,
and the desired energy estimate
\begin{equation*}
\E(\norm{u_0}_{H^{1+\sigma}(D)}^2) \leq  \tilde C_0\;\E(\norm{f}_{H^{-1+\sigma}(D)}^2),
\end{equation*}
follows directly by the elliptic regularity theory \cite{LM}.

We show the case of $n\ge 1$ by induction. Assume that \eqref{un_est} holds
for $n = 0, 1, \cdots, l-1$, then for the source term in \eqref{un1_pde}, it follows that $\nabla \cdot (\eta\nabla u_{l-1})\in L^2(\Ome, H^{-1+\sigma}(D))$. 
By the Lax-Milgram theorem, there exists $u_l \in L^2(\Ome, H_0^1(D))$ solving
\eqref{un1_pde} for $n = l$.
Let $C_0=\tilde C_0(1+b_0^2)$, by the elliptic regularity theory \cite{LM}, we get
\begin{eqnarray*}
\E(\norm{u_l}_{H^{1+\sigma}(D)}^2) & \le & \tilde C_0 \; \E(\norm{\nab\cdot(\eta \nab u_{l-1})}_{H^{-1+\sigma}(D)}^2) \\
                          & \le & \tilde C_0 \left( \E(\norm{\nab\eta\cdot\nab u_{l-1})}_{H^{-1+\sigma}(D)}^2) + \E(\norm{\eta\Delta u_{l-1}}_{H^{-1+\sigma}(D)}^2) \right)\\
                          & \le & \tilde C_0 b_0^2 \; \E(\norm{u_{l-1}}_{H^{1+\sigma}(D)}^2)\\
                          & \le & C_0^{l+1} \;\E(\norm{f}_{H^{-1+\sigma}(D)}^2).
\end{eqnarray*}
This completes the proof.
\end{proof}

A more practical and interesting mode expansion for the solution is given by its finite-terms 
approximation.  Namely, for a non-negative integer $N$, we define the partial sum
\begin{equation}\label{u_exp_finite}
 U^\veps_N(\omega,x):= \sum_{n=0}^{N-1} \veps^n u_n(\omega,x),
\end{equation}
and its associated residual
\begin{equation}\label{rN}
 r_{N}^{\veps}(\omega,x) := u^{\veps}(\omega,x) - U^\veps_N(\omega,x).
\end{equation}
For a given $N$, an upper bound for the residual $r_{N}^{\veps}$ is established by the 
following theorem.

\begin{thm}\label{err_rN}
Assume that $\veps<1$ and $f\in L^2(\Ome, H^{-1+\sigma}(D))$ for $\sigma\in(0,1]$. 
Let $r_{N}^{\veps}$ be the residual defined above. Then
\begin{equation}\label{rN_est}
\E(\norm{r_{N}^{\veps}}_{H^{1+\sigma}(D)}^2) \leq  C_0^{N+1} \veps^{2N} \;\E(\norm{f}_{H^{-1+\sigma}(D)}^2)
\end{equation}
for some positive constant $C_0$ independent of $N$ and $\veps$.
\end{thm}

\begin{proof} 
By a direct comparison,  it is easy to check that $r_{1}^{\veps}=u^{\veps}-u_0$ satisfies
\begin{alignat*}{2}
-\nabla \cdot (a_0(\omega,\cdot)\nabla r_{1}^{\veps}(\omega,\cdot) ) 
& = \veps\; \nabla \cdot \left(\eta(\omega,\cdot)\nabla u^\veps (\omega,\cdot)\right)   
&&\qquad \mbox{in} \; D,\\
r_{1}^{\veps}(\omega,\cdot) &= 0  &&\qquad \mbox{on} \; \partial D.\\
\end{alignat*}
Therefore,
\begin{eqnarray*}
\E(\norm{r_{1}^{\veps}}_{H^{1+\sigma}(D)}^2) &\leq & \tilde C_0 \;\veps^2 \;\E(\norm{\nab\cdot(\eta \nab u^\veps)}_{H^{-1+\sigma}(D)}^2)  \\
                          &\leq & \tilde C_0 b_0^2\; \veps^2 \; \E(\norm{u^\veps}_{H^{1+\sigma}(D)}^2)  \\
                          &\leq & C_0^2 \; \veps^2 \;\E(\norm{f}_{H^{-1+\sigma}(D)}^2),
\end{eqnarray*}
where $C_0=\tilde C_0(1+b_0^2)$. Assume that \eqref{rN_est} holds for $n = 1, \cdots, l-1$.
For $n=l$, it can be shown that $r_{l}^{\veps}$ is the solution of
\begin{alignat*}{2}
-\nabla \cdot (a_0(\omega,\cdot)\nabla r_{l}^{\veps}(\omega,\cdot) ) & = \veps\; \nabla \cdot \left(\eta(\omega,\cdot)\nabla r_{l-1}^{\veps}(\omega,\cdot)\right)   &&\qquad \mbox{in} \; D,\\
r_{l}^{\veps}(\omega,\cdot)  &= 0  &&\qquad \mbox{on} \; \partial D.
\end{alignat*}
A parallel argument as above yields the desired estimate
$$ \E(\norm{r_{l}^{\veps}}_{H^{1+\sigma}(D)}^2) \leq C_0 \veps^2 \; \E(\norm{r_{l-1}^{\veps}}_{H^{1+\sigma}(D)}^2)  \leq  C_0^{l+1} \veps^{2l} \;\E(\norm{f}_{H^{-1+\sigma}(D)}^2). $$ 
The proof is complete.
\end{proof}

In particular, by letting $N\to\infty$, we obtain the convergence of the partial sum $U^\veps_N$:
\begin{cor}\label{err_rN_inf}
Let $r_{N,\varepsilon}$ be the residual defined in \eqref{rN}. 
If $\veps < \min\left\{1, \dfrac{1}{\sqrt{C_0}}\right\}$, then $\E(\norm{r_{N}^{\veps}}_{H^{1+\sigma}(D)}^2) \to 0 $ as $N\to\infty$.
\end{cor}

Corollary \ref{err_rN_inf} shows that the expansion (\ref{u_exp}) is valid and the partial 
sum $U^\veps_N$ given by \eqref{u_exp_finite} converges to $u_{\eps}$ as $N \to\infty$, 
as long as $\veps$ is sufficiently small.

\section{Multi-modes Monte Carlo method}\label{sec-4}

\subsection{Numerical algorithm and computational complexity}\label{sec-4.1}
We introduce the multi-modes Monte Carlo method for approximating
the solution of the problem \eqref{bvp}--\eqref{bnd_cond}. The method is based upon
the multi-modes representation \eqref{u_exp} and its finite-terms approximation \eqref{u_exp_finite}.
For each mode $u_n$, the standard finite difference or finite element method may be 
applied to discretize the elliptic partial differential equations \eqref{u0_pde}--\eqref{un1_pde}, 
and the classical Monte Carlo method is employed for sampling
the probability space and for computing the statistics of the numerical solution. Here, 
we introduce the algorithm wherein the finite element method is used for solving the elliptic PDEs.  

Let $M$ be a large positive integer which denotes the number of realizations for the Monte Carlo method.
$\cT_h$ stands for a quasi-uniform partition of $D$ such that
$\overline{D}=\bigcup_{K\in\cT_h} \overline{K}$. Let $h:=\mbox{max}\{h_K; K\in\cT_h\}$,
wherein $h_K$ is the diameter of $K\in \cT_h$, and $P$ be number of the degrees of freedom
associated with the triangulation $\cT_h$ in each direction.
Let $V_r^h$ be the standard finite element space defined by
\[ 
V_r^h:=\{v\in H_0^1(D);\; v|_K \; \mbox{is a polynomial of degree $r$ for each} \; K\in\cT_h \}. 
\]
For each $j=1,2,\cdots, M$, we sample i.i.d. realizations of the source function
$f(\omega_j,\cdot)$ and random medium coefficient $\eta(\omega_j,\cdot)$.
The finite element solutions for the mode $u_n^h(\omega_j,\cdot)$ are obtained recursively as follows:
\begin{alignat}{2}
\bigl(a_0 \nab u_0^h(\ome_j,\cdot), \nab v^h)_{D} 
&= \langle f(\ome_j,\cdot), v^h \rangle_{D} &&\quad\forall v^h\in V_r^h,  \label{u0h_bvp} \\
\bigl(a_0 \nab u_n^h(\ome_j,\cdot), \nab v^h)_{D} &= - \big( \eta(\ome_j,\cdot) \nab u_{n-1}^h(\ome_j,\cdot), \nab v^h \bigl)_{D} &&\quad \forall v^h\in V_r^h \label{un1h_bvp}
\end{alignat}
for $n\ge 1$.
An application of the Lax-Milgram theorem and an induction argument for the variational problems
\eqref{u0h_bvp} and \eqref{un1h_bvp} yields the following energy estimates for the finite element 
solution $u_n^h(\ome_j,\cdot)$.

\begin{thm}\label{unh_est}
 If $f\in L^2(\Ome, H^{-1+\sigma}(D))$ for $\sigma\in(0,1]$, there holds for $n\geq 0$
\begin{equation}
\E\bigr(\norm{u_n^h}^2_{H^1(D)}\bigr) \leq  C_0^{n+1} \;\E(\norm{f}_{H^{-1+\sigma}(D)}^2)
\end{equation}
for some constant $C_0$ independent of $n$ and $\veps$.
\end{thm}

We then approximate the expectation $\E(u_n)$ of each mode $u_n$ by the sampling average  
$\frac{1}{M}\sum_{j=1}^M u_n^h(\omega_j,\cdot)$. Consequently, by virtue of \eqref{u_exp_finite}, 
the algorithm yields a finite-modes approximation of $\E(u^\veps)$ given by
\begin{equation}\label{Psi_N}
\Psi^h_N = \dfrac{1}{M}\sum_{j=1}^M \sum_{n=0}^{N-1} \veps^n u_n^h(\omega_j,\cdot).
\end{equation}

Very importantly, it is observed from \eqref{u0_pde}--\eqref{un1_pde} that 
all the modes share the same deterministic elliptic operator $-\nabla \cdot (a_0 \nabla)$
and the bilinear forms in  \eqref{u0h_bvp}--\eqref{un1h_bvp} are identical.
Using this crucial fact, it turns out that an $LU$ direct solver for the discretized equations 
\eqref{u0h_bvp}--\eqref{un1h_bvp} leads to a tremendous saving in the computational costs.
More precisely, we first compute an $LU$ decomposition for the associated matrix
of the bilinear form $\bigl(a_0 \nab u_n^h(\ome_j,\cdot), \nab v^h)_{D}$.
The resulting lower and upper triangular matrices, $L$ and $U$, are stored and used repeatedly
to obtain the solutions for all modes and all samples
by simple forward and backward substitutions. This speeds up the sampling tremendously, since 
in contrast to a complete linear solver with $O(P^{3d})$ computational complexity, only an $O(P^{2d})$
computational complexity is involved to calculate one single sample by the use of forward and backward substitutions.
Here $d$ denotes the spatial dimension of the domain $D$.
The precise description of this procedure is given in the following algorithm. \\

\medskip
\noindent
{\bf Main Algorithm}
\smallskip

\begin{description}
\item Inputs: $f, \eta, \veps, h, M, N.$
\item Set $\Psi^h_N(\cdot)=0$ (initializing).
\begin{description}
\item For $j=1,2,\cdots, M$
\item Set $U^h_N(\omega_j,\cdot)=0$ (initializing).
\begin{description}
\item For $n=0,1,\cdots, N-1$
\item Solve for $u^h_n(\omega_j,\cdot) \in V^h_r$ such that
\begin{eqnarray*}
&& \bigl(a_0 \nab u_0^h(\ome_j,\cdot), \nab v^h)_{D} = \langle f(\ome_j,\cdot), v^h \rangle_{D} \quad\forall v_h\in V_r^h,  \\
&& \bigl(a_0 \nab u_n^h(\ome_j,\cdot), \nab v^h)_{D} = - \bigl( \eta(\ome_j,\cdot) \nab u_{n-1}^h(\ome_j,\cdot), \nab v^h \bigl)_{D} \\
&& \qquad \forall v_h\in V_r^h, \quad \mbox{if} \; n\ge 1.
\end{eqnarray*}
\item Set $U^h_N(\omega_j,\cdot)\leftarrow U^h_N(\omega_j,\cdot) +\veps^n u^h_n(\omega_j,\cdot)$.
\item End For
\end{description}
\item Set $\Psi^h_N(\cdot) \leftarrow \Psi^h_N(\cdot) +\frac{1}{M} U^h_N(\omega_j,\cdot)$.
\item End For
\end{description}
\item Output $\Psi^h_N(\cdot)$.
\end{description}

\medskip

The whole algorithm requires one to solve a total of $MN$ linear systems for 
$N$ modes and $M$ realizations for each mode.  Since all linear systems share the same 
coefficient matrix, we only need to perform one $LU$ decomposition of the
matrix and save the lower and upper triangular matrices. The decomposition is then reused to solve the
remaining $MN-1$ linear systems by performing $MN-1$ sets of forward
and backward substitutions. It is straightforward that the computational cost of the whole algorithm
is $O(\frac{3}{2}P^{3d})+O(MNP^{2d})$.  In light of Theorem \ref{err_rN}, a relatively small 
number $N$ of modes is needed to get desired accuracy in practice, since the associated 
residual $r_{N}^{\veps}$ has an order of $\veps^N$. Hence we may regard $N$ as a constant.
To get the same order of errors for the finite element approximation and the Monte Carlo simulation
(see Section \ref{sec-4.2} and \ref{sec-5}), we may choose $M\sim O(P^4)$. Consequently, 
the total cost for implementing the algorithm becomes $O(\frac{3}{2}P^{3d})+O(NP^{2d+4})$.  
As a comparison, a brute force Monte Carlo method for solving the problem \eqref{bvp}--\eqref{bnd_cond} 
with the same number of realization gives rise to $O(\frac{3P^{3d+4}}2)$ multiplications/divisions.
It is seen that the computational cost of the proposed algorithm is significantly reduced by 
the use of the multi-modes expansion and by using the $LU$ decomposition matrices repeatedly.

\subsection{Convergence analysis}\label{sec-4.2}
In this subsection, we derive the error estimates for the proposed algorithm.
First, it is observed that
$\E(u^\veps)-\Psi^h_N$ can be decomposed as
\begin{equation}\label{err_decomp}
\bigl(\E(u^\veps)-\E(U^\veps_N)\bigr)
+ \bigl( \E(U^\veps_N)- \E(U^h_N)\bigr) +\bigl( \E(U^h_N)-\Psi^h_N \bigr),
\end{equation}
where $U^\veps_N$ and $\Psi^h_N$ are given by \eqref{u_exp_finite} and \eqref{Psi_N} respectively, and 
\begin{align*}
	U^h_N(\omega,x):= \sum_{n=0}^{N-1} \veps^n u_n^h(\omega,x).
\end{align*}
It is clear that the first term in the decomposition \eqref{err_decomp} measures the error due 
to the finite-modes expansion, the second term is the spatial discretization error, and the
third term represents the statistical error due to the Monte Carlo method.

The finite-modes representation error is given in Theorem \ref{err_rN}. That is,
\begin{equation}\label{err_first_term}
\E\bigr(\norm{u^\veps -U^\veps_N}_{H^{1+\sigma}(D)}^2\bigr)
\leq  C_0^{N+1} \veps^{2N} \;\E\bigr(\norm{f}_{H^{-1+\sigma}(D)}^2\bigr).
\end{equation}
Let $\Phi^h_n= \frac{1}{M} \sum_{j=1}^M u_n^h(\omega_j,\cdot)$, then it is clear that
$$
\E(U^h_N)-\Psi^h_N =\sum_{n=0}^{N-1} \veps^n \bigl(\E(u^h_n) -\Phi^h_n \bigr).
$$
With the standard error estimates for the Monte Carlo method (cf. \cite{Babuska_Tempone_Zouraris04, Liu_Riviere13}), the statistical error can be bounded as follows: 
\begin{eqnarray*}
\E\bigl(\norm{\E(U^h_N)-\Psi^h_N}^2_{H^1(D)}\bigr) 
&\leq & 2\sum_{n=0}^{N-1} \veps^{2n} \; \E\bigl(\|\E(u^h_n) -\Phi^h_n\|^2_{H^1(D)} \bigr) \\
&\leq & \frac{2}{M} \sum_{n=0}^{N-1} \veps^{2n} \; \E\bigr(\norm{u_n^h}^2_{H^1(D)}\bigr).
\end{eqnarray*}
From Theorem \ref{unh_est}, by choosing $\veps \leq \min \left \{1, \frac{1}{\sqrt{C_0}} \right\}$, we have
\begin{eqnarray}\label{err_third_term_H1}
\E\bigl(\norm{\E(U^h_N)-\Psi^h_N}^2_{H^1(D)}\bigr)
&\leq & \frac{2C_0}{M} \left(\sum_{n=0}^{N-1} \veps^{2n} C_0^n \right) \E\bigr(\norm{f}^2_{H^{-1+\sigma}(D)}\bigr) \nonumber \\
&\leq & \frac{2C_0}{(1-C_0\veps^2)M} \;\E\bigr(\norm{f}^2_{H^{-1+\sigma}(D)}\bigr).
\end{eqnarray}

In order to estimate the spatial discretization error $\E(U^\veps_N)- \E(U^h_N)$, for each mode, 
let us define an auxiliary function $\tilde{u}^h_n\in V^h_r$ as the solution of the following 
discrete problem: 
\begin{equation}\label{var_pro_uhn_tilde}
\bigl(a_0 \nab\tilde{u}^h_n(\ome_j,\cdot), \nab v^h)_{D} 
= - \bigl(\eta(\ome_j,\cdot) \nab u_{n-1}(\ome_j,\cdot), \nab v^h)_{D}
\end{equation}
for all $v_h\in V_r^h$ and $n\ge1$.
For simplicity, we restrict ourselves to the case of $r=1$. Namely, $v^h$ is a linear polynomial 
on each $K\in\cT_h$.  The case of $r>1$ can be derived similarly, and we omit it for the clarity 
of the exposition.  The standard error estimation technique for the finite element method 
(cf. \cite{Brenner_Scott}) and the energy estimate \eqref{un_est} yield
\begin{eqnarray}\label{err_unh1}
\E\bigl(\norm{u_n-\tilde{u}^h_n}_{H^1(D)}\bigr) &\leq & C h^{\sigma}\; \E\bigl(\norm{u_n}_{H^{1+\sigma}(D)}\bigr) \nonumber \\
&\leq& C \big(\sqrt{C_0} \big)^{n+1}  h^\sigma\; \E\bigl(\norm{f}_{H^{-1+\sigma}(D)}\bigr). \label{err_unh1_H1}
\end{eqnarray}
Next we estimate the error $\E\bigl(\norm{\tilde{u}_n^h-u_n^h}_{H^1(D)}\bigr)$. 

Recall that for $n\geq 1$, $u^h_n\in V^h_r$ is defined by
\begin{equation}\label{var_pro_uhn}
\bigl(a_0 \nab u_n^h(\ome_j,\cdot), \nab v^h)_{D} = - \bigl(\eta(\ome_j,\cdot) \nab u_{n-1}^h(\ome_j,\cdot), \nab v^h)_{D}
\quad\forall v_h\in V_r^h.
\end{equation}
For each fixed sample $w=w_j$, a direct comparison of \eqref{var_pro_uhn_tilde} and \eqref{var_pro_uhn} gives rise to
$$ \bigl(a_0 (\nab \tilde{u}_n^h- \nab u_n^h), \nab v^h)_{D} = - \bigl(\eta (\nab u_{n-1}-\nab u_{n-1}^h), \nab v^h)_{D}
\quad\forall v_h\in V_r^h. $$
By setting $v_h=\tilde{u}_n^h- u_n^h$ and using the Cauchy-Schwarz inequality, it follows that
\begin{equation}\label{est_CS}
\E\bigl(\norm{\nab\tilde{u}_n^h-\nab u_n^h}_{L^2(D)}\bigr) \le \beta\; \E\bigl(\norm{\nab u_{n-1}-\nab u_{n-1}^h}_{L^2(D)}\bigr),
\end{equation}
where $\beta=\dfrac{1}{\min_{x\in \bar{D}}{a_0(x)}}$.
An application of the Poincar\'e-Friedrichs inequality leads to the estimates
\begin{eqnarray}\label{err_unh2}
\E\bigl(\norm{\tilde{u}_n^h-u_n^h}_{H^1(D)}\bigr) & \le & \beta_1\; \E\bigl(\norm{\nab u_{n-1}-\nab u_{n-1}^h}_{L^2(D)}\bigr). \label{err_unh2_H1}
\end{eqnarray}
Here $\beta_0$ and $\beta_1$ are suitable constants depending on $a_0(x)$ and the domain $D$ only.

In light of  \eqref{err_unh1_H1} and \eqref{err_unh2_H1}, we see that
\begin{eqnarray*}
\E\bigl(\norm{u_n-u_n^h}_{H^1(D)}\bigr)  & \leq &  C \big( \sqrt{C_0} \big)^{n+1} h^\sigma\; \E\bigl(\norm{f}_{H^{-1+\sigma}(D)}\bigr) +   \beta_1\; \E\bigl(\norm{\nab u_{n-1}-\nab u_{n-1}^h}_{L^2(D)}\bigr) \\
& \leq &  C \big( \sqrt{C_0} \big)^{n+1} h^\sigma\; \E\bigl(\norm{f}_{H^{-1+\sigma}(D)}\bigr) +   \beta_1\; \E\bigl(\norm{u_{n-1}-u_{n-1}^h}_{H^1(D)}\bigr).
\end{eqnarray*}
By applying the above inequality recursively, it is obtained that
\begin{align}\label{err_unh_rec}
\E\bigl(\norm{u_n-u_n^h}_{H^1(D)}\bigr)   &\leq   C h^\sigma\; \E\bigl(\norm{f}_{H^{-1+\sigma}(D)}\bigr)  \sum_{j=0}^{n-1} \beta_1^j \big(\sqrt{C_0}\big)^{n+1-j} \\  
& \qquad +   \beta_1^n\; \E\bigl(\norm{u_{0}-u_{0}^h}_{H^1(D)}\bigr) \notag.
\end{align}
Note that $u_0$ and $u_0^h$ solves \eqref{u0_pde} and \eqref{u0h_bvp} respectively, hence
\begin{equation}\label{err_u0h_rec}
 \E\bigl(\norm{u_{0}-u_{0}^h}_{H^1(D)}\bigr)  \le C h^\sigma\; \E\bigl(\norm{u_0}_{H^{1+\sigma}(D)}\bigr) \leq
C \sqrt{C_0} h^\sigma \; \E\bigl(\norm{f}_{H^{-1+\sigma}(D)}\bigr).
\end{equation}
We arrive at
\begin{equation}\label{err_unh}
\E\bigl(\norm{u_n-u_n^h}_{H^1(D)}\bigr)   \leq   C h^\sigma\; \E\bigl(\norm{f}_{H^{-1+\sigma}(D)}\bigr)  \sum_{j=0}^{n} \beta_1^j \big(\sqrt{C_0} \big)^{n+1-j}
\end{equation}
by substituting  \eqref{err_u0h_rec} into \eqref{err_unh_rec}. Correspondingly,
\begin{eqnarray}\label{err_seond_term_H1}
\E\bigl(\norm{U_N^\veps-U_N^h}_{H^1(D)}\bigr)  & \leq &  C h^\sigma\; \E\bigl(\norm{f}_{H^{-1+\sigma}(D)}\bigr)  \sum_{n=0}^{N-1}\sum_{j=0}^{n} \veps^{n}\beta_1^j \big( \sqrt{C_0} \big)^{n+1-j} \nonumber \\
& \leq & C_1(\veps,N) \;h^\sigma\; \E\bigl(\norm{f}_{H^{-1+\sigma}(D)}\bigr),
\end{eqnarray}
where $C_1(\veps,N):=\dfrac{C \sqrt{C_0}}{\sqrt{C_0}-\beta_1}\left[\sqrt{C_0}\dfrac{1-(\veps \sqrt{C_0})^N}{1-\veps \sqrt{C_0}}-\beta_1\dfrac{1-(\veps \beta_1)^N}{1-\veps \beta_1}\right]$.\\

Combining \eqref{err_first_term}, \eqref{err_third_term_H1}, and \eqref{err_seond_term_H1}, 
we get the following error estimate for the full algorithm.

\begin{thm}\label{err_mme}
For a given source function $f\in L^2(\Ome, H^{-1+\sigma}(D))$ with $\sigma\in(0,1]$,
let $\Psi_N^h$ be the numerical solution obtained in the Main Algorithm
with $r=1$. There holds
\begin{equation*}
\E\bigl( \|\E(u^\veps) - \Psi_N^h\|_{H^1(D)}\bigr) \leq C  (\veps^N +h^\sigma + M^{-\frac12}) \; \E\bigl(\norm{f}_{H^{-1+\sigma}(D)}\bigr).
\end{equation*}
for some positive constant $C$ independent of $\veps$, $h$, $M$ and $N$.
\end{thm}

\begin{rmk}
For the $L^2$-norm error $\E\bigl( \|\E(u^\veps) - \Psi_N^h\|_{L^2(D)}\bigr)$ of the whole algorithm, 
it is expected that an order of $O\left(\veps^N +h^{1+\sigma} + M^{-\frac12}\right)$ can be achieved, 
as predicted by the numerical results in Section \ref{sec-5} (see Table \ref{table:con_order_h_1d}).
\end{rmk}

\section{Numerical experiments}\label{sec-5}
In this section, we present a series of numerical experiments to illustrate the accuracy 
and efficiency of the proposed method.  Section \ref{sec-5.1} studies the accuracy of the method 
for solving one-dimensional problems, where the analytical solution is known and hence can be 
used for comparison. The application of the numerical algorithm to 
two-dimensional problems is  elaborated in Section \ref{sec-5.2}.

\subsection{One-dimensional examples}\label{sec-5.1}
We consider the following boundary value problem:
\begin{align*}
-\frac{d}{dx}\Bigl( \bigl(1+\veps Y(\omega) \bigr)  \frac{du^\veps (\omega,x)}{dx}\Bigr) 
&= Y(\omega), \qquad 0<x<1, \\
u^\veps(\omega,0) = 0,  \quad u^\veps(\omega,1) &= 0,
\end{align*}
where $Y(\omega)$ is a uniformly distributed random variable over $[0,1]$. The analytical 
solution for the boundary value problem takes the form 
$u(x,\omega)=\dfrac{Y(\omega)}{2(1+\veps Y(\omega))}(x-x^2)$, and its expectation is 
$\E(u^\veps) = \dfrac{1}{2}\left(\dfrac{1}{\veps} - \dfrac{1}{\veps^2}\ln(1+\veps) \right) (x-x^2)$.

To test the validity of the multi-modes expansion and the accuracy of the numerical algorithm, 
we set $h=0.01$ for the spatial discretization and $M=10^6$  for the number of realizations. 
Table \ref{table:MME_Error_1d} displays the accuracy of the approximation for various $\veps$ 
and different number of modes, where the relative $L^2$-norm is defined as 
$\|\E(u^\veps) - \Psi_N^h\|_{L^2(D)} /\|\E(u^\veps) \|_{L^2(D)}$.  It is observed that the 
multi-modes Monte-Carlo finite element method gives rise to accurate approximation as long as 
the magnitude of the random perturbation is not large. As expected, more modes are required 
to suppress the errors as the magnitude of the perturbation $\veps$ increases.

\begin{table}[!htbp]
\centering
\begin{tabular}{| c || c | c | c | c | c |}
	\hline
	$\veps$ & $N = 2$ & $N = 3$ & $N = 4$ & $N = 5$  & $N=6$ \\
	\hline
	\hline
	$0.2$ & $1.95 \times10^{-2} $ & $ 3.15 \times10^{-3} $ & $4.74 \times10^{-4} $ & $1.45\times10^{-4}$  & $ 6.40\times10^{-5}$ \\
	\hline
	$0.4$ & $7.66 \times10^{-2} $ & $ 2.42 \times10^{-2} $ & $8.05 \times10^{-3} $ & $2.71\times10^{-3}$  & $9.84\times10^{-4}$\\
	\hline
	$0.6$ & $0.1688$ & $0.0806$ & $0.0391$ & $0.0208$  & $0.0100$  \\
	\hline
	$0.8$ & $ 0.2960$ & $0.1869$ & $0.1222$ & $0.0839$ & $0.0574$  \\
	\hline
\end{tabular}
\caption{Relative $L^2$-norm error for the multi-modes Monte Carlo finite element 
approximation $\Psi^h_N$ with different $\veps$ and $N$.}\label{table:MME_Error_1d}
\end{table}

Next we study the convergence rate of the proposed algorithm numerically.   
Note that the whole error consist of three parts as given in \eqref{err_decomp}.
The statistical error term arising from the Monte Carlo method is standard and we omit here.
In order to test the error term associated with the spatial discretization, we use large 
numbers of Monte Carlo realizations and adopt high-order mode expansion
such that the total error of the whole algorithm is dominated by the spatial discretization error. 
To this end, we fix $\veps=0.5$ in the following and set $N=10$, $M=10^6$ respectively. 
The $H^1$ and $L^2$-norm errors for the multi-modes Monte Carlo finite element approximation $\Psi^h_N$
are shown in Table \ref{table:con_order_h_1d}. It is observed that a convergence rate of $O(h)$
is obtained for the numerical solution with respect to the $H^1$-norm. 
Note that $f\in L^2(\Omega,L^2(D))$ in this example, hence the numerical convergence rate 
is consistent with the theoretical one as obtained in Theorem \ref{err_mme}.
Furthermore, it is seen that the $L^2$-norm error exhibits a convergence rate of $O(h^2)$.

\begin{table}[!htbp]
\centering
\begin{tabular}{| c || c | c || c | c |}
	\hline
	$ h $ & $ \|\E(u^\veps) - \Psi_N^h\|_{H^1(D)} $ & order & $\|\E(u^\veps) - \Psi_N^h\|_{L^2(D)} $ & order \\
	\hline
	\hline
	$0.2$ &  $2.19\times10^{-1}$ & & $1.38 \times 10^{-3}$ &    \\
	\hline
	$0.1$ & $ 1.09 \times 10^{-2}$ & $1.00$  &  $3.29 \times 10^{-4}$ & $2.07$    \\
	\hline
	$0.05$ & $ 5.46 \times 10^{-3} $ & $1.00$ & $7.43 \times 10^{-5}$ & $2.15$    \\
	\hline
	$0.025$  & $2.73 \times 10^{-3}$ & $1.00$ & $1.28 \times 10^{-5}$ & $2.53$   \\
	\hline
\end{tabular}
\caption{ $H^1$ and $L^2$-norm errors for the multi-modes Monte Carlo finite element 
approximation $\Psi^h_N$ with decreasing $h$, and the corresponding numerical convergence orders.
}\label{table:con_order_h_1d}
\end{table}

\begin{figure}[!htbp]
\centering
\includegraphics[height=4.5cm,width=7cm]{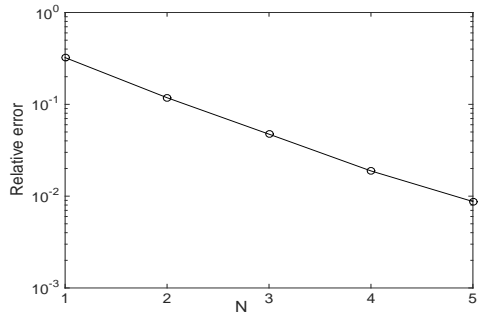}
\caption{Relative $H^1$-norm error for the multi-modes Monte Carlo finite element approximation 
$\Psi^h_N$ when different number of modes is used; $\veps=0.5$.}\label{fig:err_sigma_05_diff_N_1d}
\end{figure}

To study the convergence rate for the finite modes expansion, we fix $\veps=0.5$ and set $h=0.01$, 
$M=10^6$ respectively. For $N\le 4$, the error due to truncation of modes becomes dominant.
Figure \ref{fig:err_sigma_05_diff_N_1d} displays the relative $H^1$-norm error for different modes.
As expected, as number of modes increases, the multi-modes Monte Carlo finite element approximation 
$\Psi^h_N$ becomes more accurate and a convergence rate of $O(\veps^{N})$ is observed.
This is consistent with the theoretical error estimation in Theorem \ref{err_mme}.

\subsection {Two-dimensional examples}\label{sec-5.2}
We consider solving the two-dimensional random elliptic problem:
\begin{alignat*}{2}
-\nabla \cdot \bigl(a(\omega,\cdot)\nabla u^\veps (\omega,\cdot)\bigr) 
&= f(\omega,\cdot) &&\qquad \mbox{in } D, \\
u^\veps(\omega,\cdot) &= 0 &&\qquad \mbox{on } \partial D,
\end{alignat*}
where the spatial domain is $D =(0,2)\times(0,2)$.
The background diffusion coefficient $a_0(x_1,x_2)=1$. The random perturbation $\eta(\omega,x)$ 
and the source function $f(\omega,x)$ are given by
\begin{align*}
\eta(\ome,x) &=0.5+0.5\sum_{m=1}^{M_\eta}\sum_{n=1}^{N_\eta} e^{-0.2(m^2+n^2)} 
\phi_{m,n}(x_1,x_2)Y_{m,n}(\ome), \\
f(\ome,x) &=x_1^2+x_2^2+\sum_{m=1}^{M_f}\sum_{n=1}^{N_f} 2e^{-0.2(m^2+n^2)} 
\psi_{m,n}(x_1,x_2)Z_{m,n}(\ome),
\end{align*}
respectively. Here $Y_{1,1}$, $\cdots$, $Y_{M_\eta,N_\eta}$
are independent uniformly distributed random variables over $[-1,1]$, and $Z_{1,1}$, $\cdots$, 
$Z_{M_f,N_f}$ are independent normally  distributed random variables with mean $0$ and variance $1$. 
The basis functions are given by 
\begin{align*}
\phi_{m,n}(x_1,x_2) &=\cos(m\pi(x_1-1))\cos(n\pi(x_2-1)), \\
\psi_{m,n}(x_1,x_2) &=\sin(m\pi(x_1-1))\sin(n\pi(x_2-1)). 
\end{align*}
We set $M_\eta=N_\eta=10$, and $M_f=N_f=5$ in the following numerical tests. 
From a simple calculation, it can be shown that $-0.6\le\eta\le1.6$ for the specified parameters.

\begin{table}[!htbp]
\centering
\begin{tabular}{| c | c |}
	\hline
	Approximation & CPU Time (s) \\
	\hline
	$\tilde{\Psi}^h$ & $3.8077 \times 10^5$ \\
	\hline
	$\Psi^h_2$ & $1.000 \times 10^4$ \\
	\hline
	$\Psi^h_3$ & $1.313 \times 10^4$ \\
	\hline
	$\Psi^h_4$ & $1.624 \times 10^4$ \\
	\hline
	$\Psi^h_5$ & $1.957\times 10^4$ \\
	\hline
\end{tabular}
\caption{CPU time required to compute the classical Monte Carlo finite element approximation 
$\tilde{\Psi}^h$ and the multi-modes Monte Carlo finite element approximation $\Psi^h_N$ .} 
\label{table:cpu_time}
\end{table}

\begin{figure}[!htbp]
\centering
\includegraphics[height=5.2cm,width=8cm]{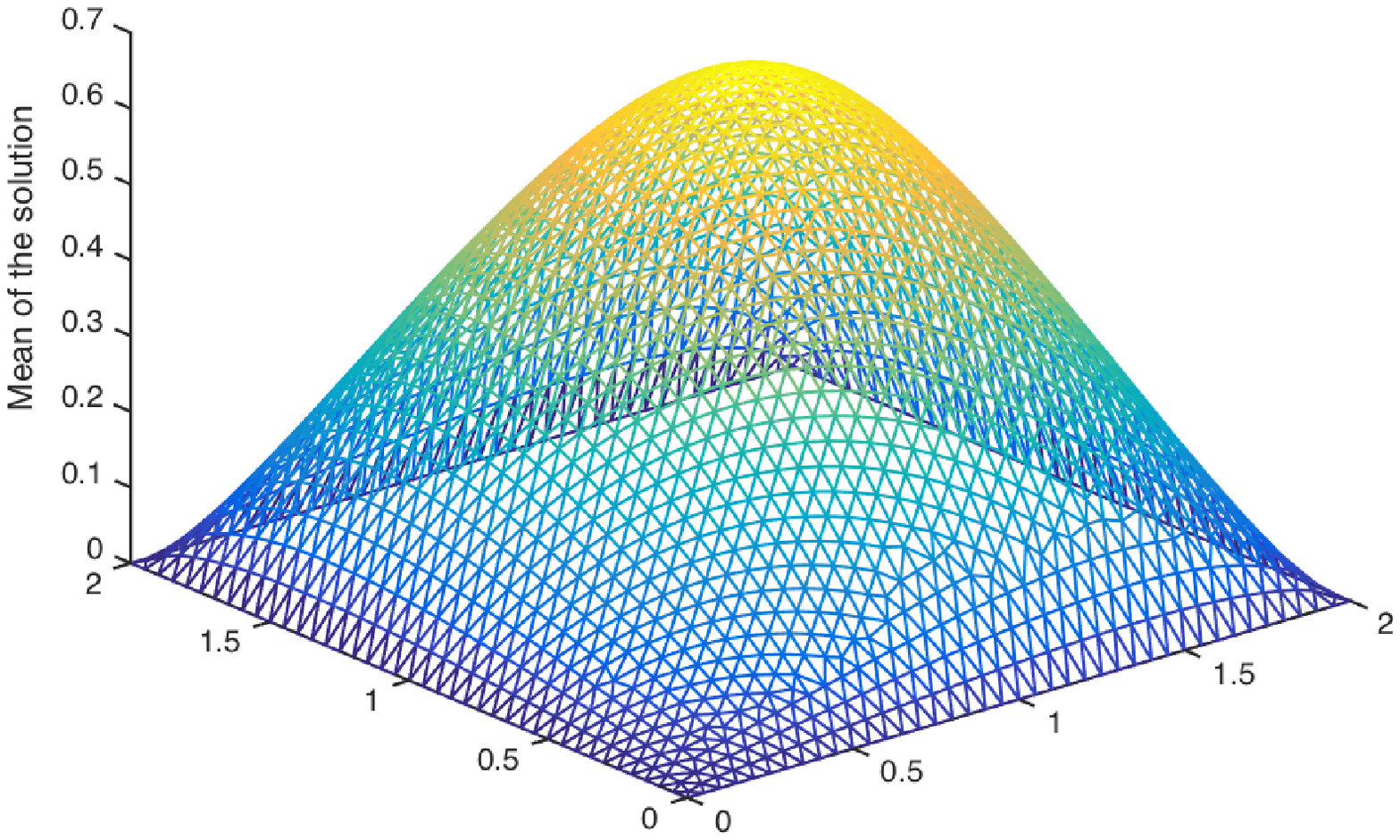}
\includegraphics[height=5.2cm,width=8cm]{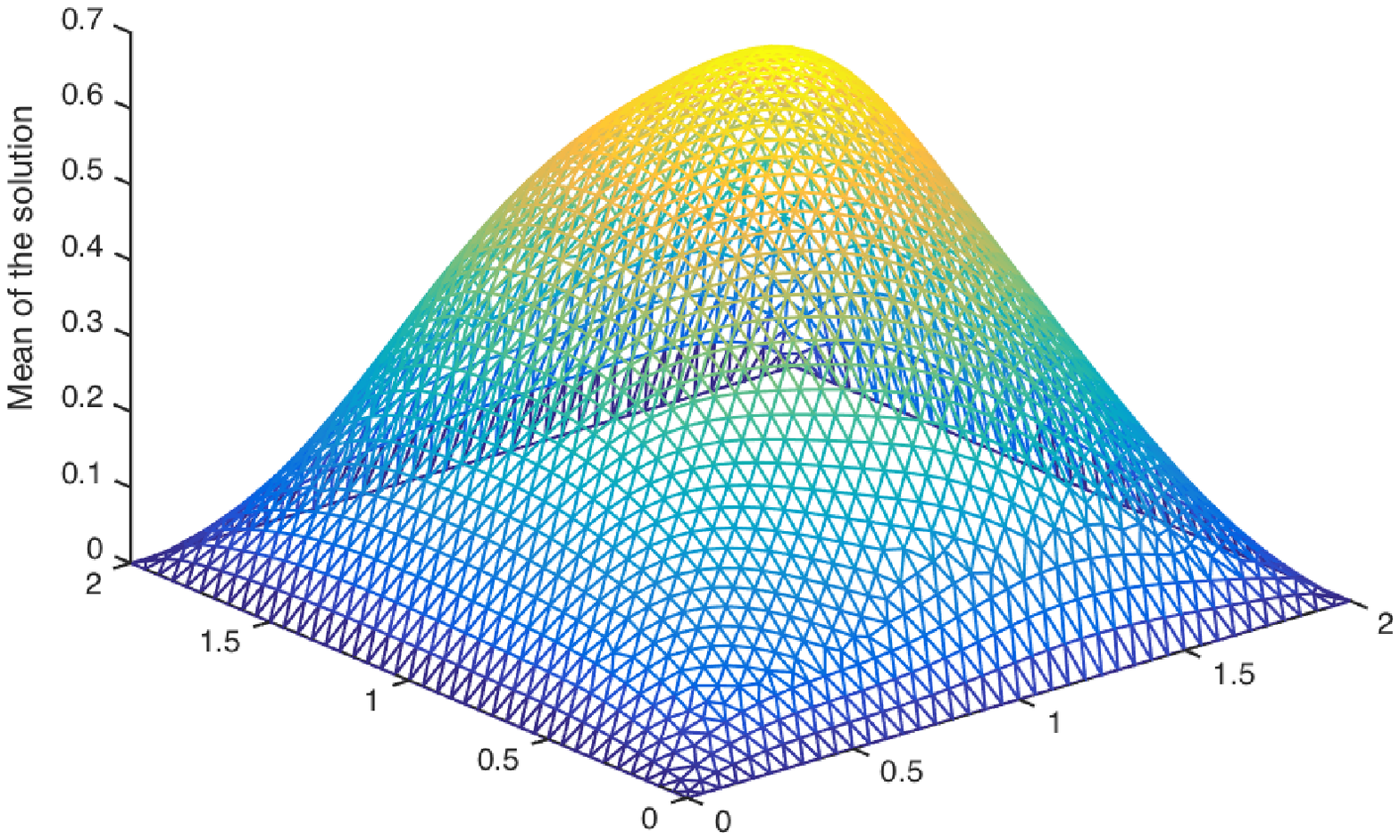}
\caption{The sample average $\Psi_5^h$ (left) and one sample $U_5^h$ (right) computed for 
$\veps = 0.5$, and $M=1000$.}
\label{fig:surf_plot_sigma_05}
\end{figure}

To partition $D$, we use a quasi-uniform triangulation $\mathcal{T}_h$ with size $h$.
The number of realizations for the Monte Carlo method is set as $M=10000$.
As a benchmark, we compare the multi-modes Monte Carlo method to 
the classical Monte Carlo finite element method (i.e. without utilizing the multi-modes expansion). 
Let us denote the numerical approximation to $\mathbb{E}(u)$ using the 
classical  Monte Carlo method by $\tilde{\Psi}^h$. 

In order to test the efficiency of the multi-modes Monte Carlo method, we set $h=0.2$ and compare 
the CPU time for computing $\Psi^h_N$ and $\tilde{\Psi}^h$.  Both methods are implemented sequentially
in Matlab on a Dell T7600 workstation.  The results of this test are shown in 
Table \ref{table:cpu_time}.  We find that the use of the multi-modes 
expansion improves the CPU time for the computation considerably.  In fact, the table shows 
that this improvement is an order of magnitude.  Also, as expected, as the number of modes 
used is increased the CPU time increases in a linear fashion.

To give an illustration of computed solutions, we show the sample average $\Psi_N^h$
and one computed sample $U_N^h$ for $\veps=0.5$ and $\veps=0.8$
in Figure \ref{fig:surf_plot_sigma_05} and Figure \ref{fig:surf_plot_sigma_08} respectively. Here
 $N=5$ is used for the multi-modes expansion and $h=0.05$ is set for the spatial discretization.

\begin{figure}[!htbp]
\centering
\includegraphics[height=5.2cm,width=8cm]{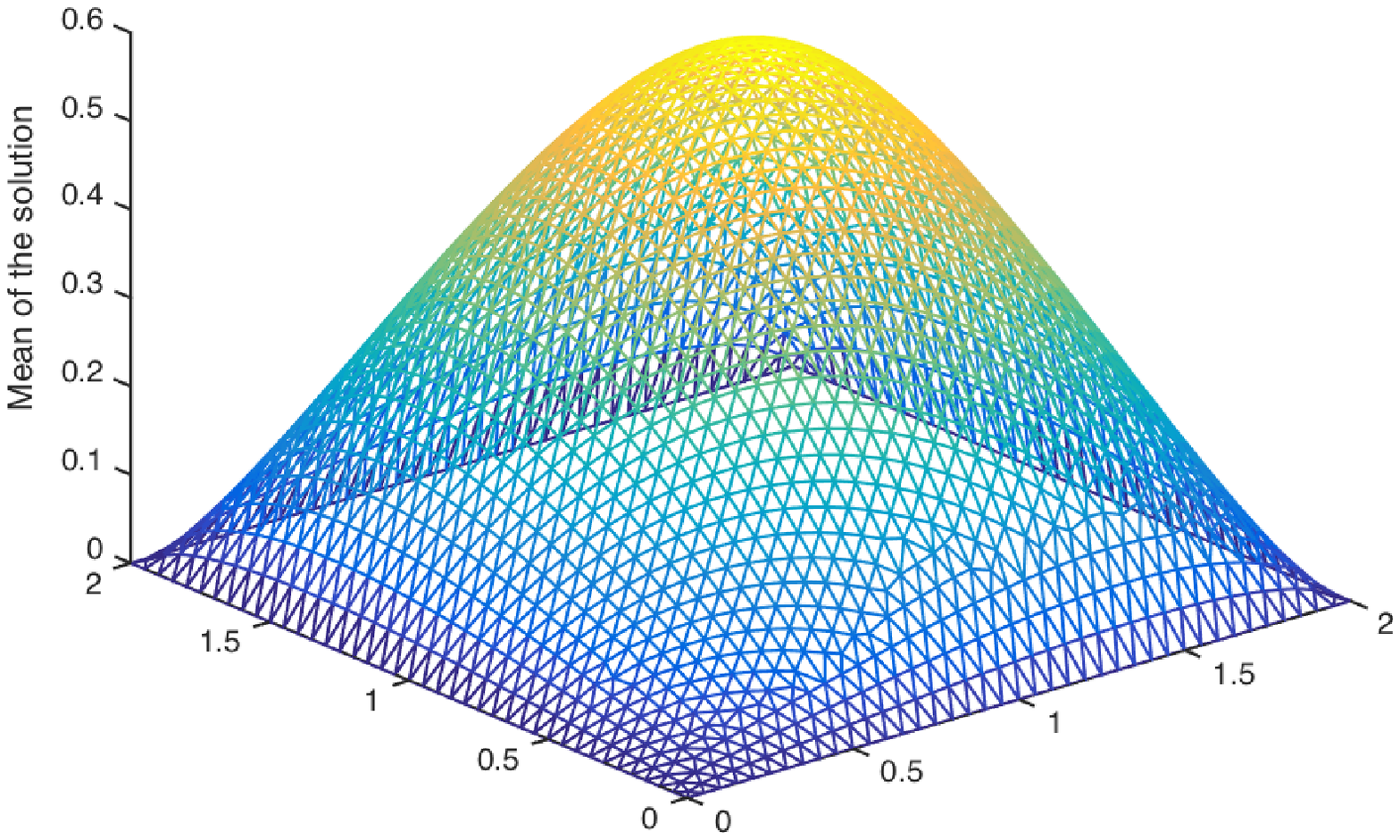}
\includegraphics[height=5.2cm,width=8cm]{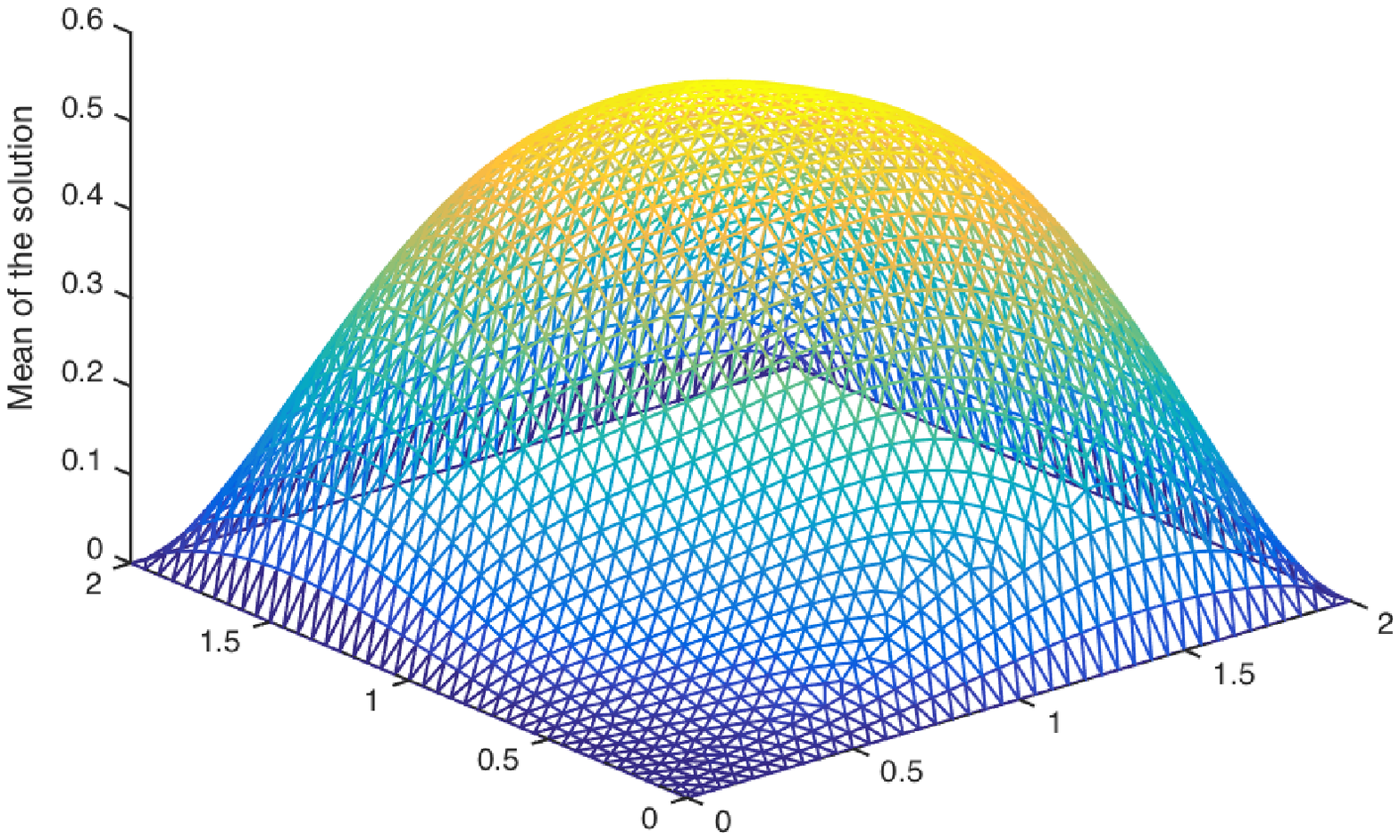}
\caption{The sample average $\Psi_5^h$ (left) and one sample $U_5^h$ (right) computed for
$\veps = 0.8$, and $M=1000$.}
\label{fig:surf_plot_sigma_08}
\end{figure}

\begin{figure}[!htbp]
\centering
\includegraphics[height=5cm,width=8cm]{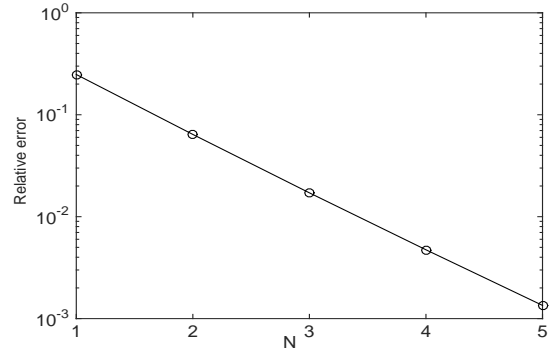}
\caption{Relative $L^2$-norm error between $\Psi^h_N$  
and $\tilde{\Psi}^h$ when different number of modes is used; $\veps=0.5$.}
\label{fig:err_sigma_05_diff_N_2d}
\end{figure}

\begin{figure}[!htbp]
\centering
\includegraphics[height=5cm,width=8cm]{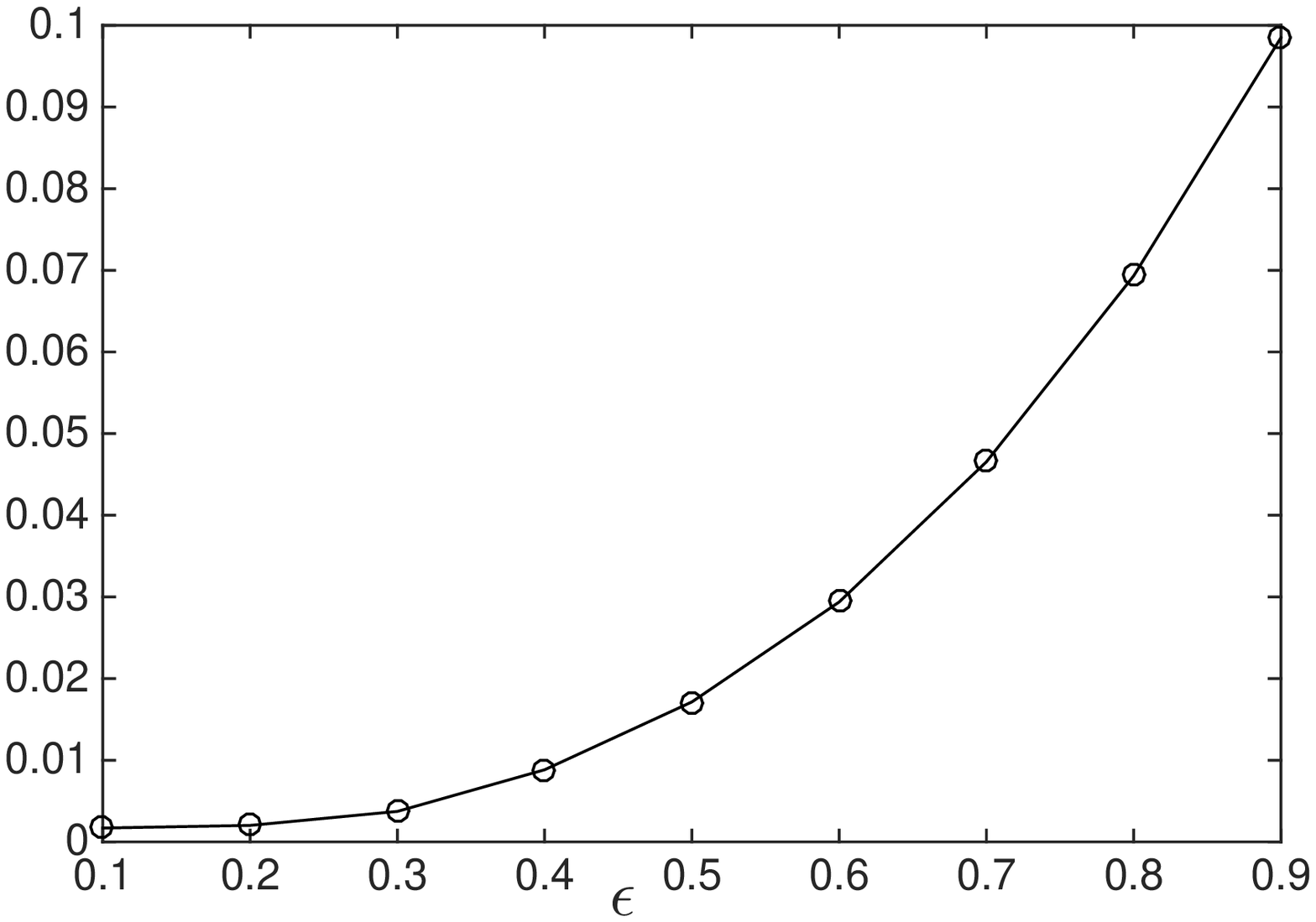}
 \caption{Relative $L^2$-norm error between $\Psi^h_N$  and $\tilde{\Psi}^h$ as $\veps$ increases. 
 $N=3$ is fixed for the multi-modes Monte-Carlo finite element method.} 
\label{fig:err_diff_sigma_N_3_2d}
\end{figure}

Next, we test the accuracy of the multi-modes Monte Carlo finite element approximation by using
the standard Monte Carlo approximation $\tilde{\Psi}^h$ as the reference. To this end,
the relative $L^2$-norm error $\| \Psi^h_N - \tilde{\Psi}^h\|_{L^2(D)}/\|\tilde{\Psi}^h\|_{L^2(D)}$ are 
computed for various $\veps$ and different numbers of modes $N$.
For clarity, we fix $h=0.05$ for the spatial triangulation in this test. If $\veps=0.5$ 
is fixed, then the relative $L^2$-norm errors for different modes are shown in 
Figure \ref{fig:err_sigma_05_diff_N_2d}.  Similar to the one-dimensional case, it is seen that 
as $N$ increases, the difference between $\Psi^h_N$ and $\tilde{\Psi}^h$ decreases steadily, 
and a rate of $O(\veps^{N})$ for the error is also observed. Moreover, if the number of modes 
used in the expansion is fixed as $N=3$, the $L^2$-norm relative errors for $\veps$ ranges 
from $0.1$ to $0.9$ are plotted in Figure \ref{fig:err_diff_sigma_N_3_2d}. We see that even 
with three modes, the multi-modes Monte-Carlo finite element method already yields accurate 
approximation as long as the magnitude of the random perturbation is not large. As expected, 
more modes are required in the expansion to obtain more accurate solutions as $\veps$ increases.
This is confirmed in Table \ref{table:MME_VS_MC_Error}, where the accuracy of the approximation
for various $\veps$ and $N$ are displayed. It is noted that the accuracy for the case of $\veps=0.2$
does not get improved  as $N$ increase from $4$ to $5$. This is due to the fact that the total 
error of the whole algorithm is dominated by the error of spatial discretization when $N=5$.

\begin{table}[!htbp]
\centering
\begin{tabular}{| c || c | c | c | c |}
	\hline
	$\veps$ & $N = 2$ & $N = 3$ & $N = 4$ & $N = 5$ \\
	\hline
	\hline
	$0.2$ & $0.0104$ & $0.0020$ & $0.0016$ & $0.0016$ \\
	\hline
	$0.4$ & $0.0416$ & $0.0088$ & $0.0026$ & $0.0016$ \\
	\hline
	$0.6$ & $0.0923$ & $0.0294$ & $0.0101$ & $0.0036$   \\
	\hline
	$0.8$ & $0.1632$ & $0.0693$ & $0.0309$ & $0.0138$   \\
	\hline
\end{tabular}
\caption{Relative $L^2$-norm error between the multi-modes Monte Carlo finite element approximation 
$\Psi^h_N$ and the classical Monte Carlo finite element approximation $\tilde{\Psi}^h$ 
for different $\veps$ and $N$.}\label{table:MME_VS_MC_Error}
\end{table}

\section{Generalization of the algorithm to general media}\label{sec-6}
To use the multi-modes Monte Carlo finite element method we developed above, it requires that 
the random media are weak in the sense that the leading coefficient $a$ in the PDE has
the form $a(\omega,x)=a_0(x) +\veps\eta(\omega,x)$ and $\veps$ is not large.
For more general random elliptic PDEs, their leading coefficients may not have 
the required ``weak form". A natural question is whether the above multi-modes Monte Carlo 
finite element method can be extended to cover these random PDEs in which
the diffusion coefficient $a(\omega,x)$ does not have the required form. A short answer to this question is yes.
The main idea for overcoming this difficulty is first to rewrite $a(x,\omega)$ into the 
desired form $a_0(x)+\veps \eta(\omega,x)$, then to apply the above ``weak" field framework.
There are at least two ways to do such a re-writing, the first one is to utilize the 
well-known Karhunen-Lo\`eve expansion and the second is to use  
a recently developed stochastic homogenization theory \cite{DGO}. Since the second approach is 
more involved and lengthy to describe, below we only outline the first approach.

In many scenarios of geoscience and material science, the random media can be described by a Gaussian random field
\cite{FGPS, Ishimaru, Lord_Powell_Shardlow}. 
Let $\overline{a}(x)$ and $C(x,y)$ denote the mean and covariance 
function of the Gaussian random field $a(\ome,x)$, respectively. Two widely used covariance functions in geoscience and materials
science are $C(x,y)=\exp(|x-y|^m/\ell)$ for $m=1,2$ and $0<\ell<1$ 
(cf. \cite[Chapter 7]{Lord_Powell_Shardlow}. Here $\ell$ is often called correlation length which 
determines the range (or frequency) of the noise.  The well-known Karhunen-Lo\`eve expansion 
for $a(\ome,x)$ takes the following form (cf. \cite{Lord_Powell_Shardlow}):
\[
a(\omega,x )= \overline{a}(x) + \sum_{k=1}^\infty \sqrt{\lambda_k} \phi_k(x) \xi_k(\omega),
\]
where $\{(\lambda_k, \phi_k)\}_{k\geq 1}$ is the eigenset of the (self-adjoint) covariance operator and
$\{\xi_k  \sim  N(0,1) \}_{k\geq 1}$ are i.i.d. random variables. It turns out in many cases there holds
$\lambda_k=O(\ell^r)$ for some $r>1$ depending on the spatial domain $D$ where the PDE is defined 
(cf. \cite[Chapter 7]{Lord_Powell_Shardlow}). Consequently, we can write
\[
a(\omega,x)=\overline{a}(x)+ \sqrt{\lambda_1} \zeta(x,\omega), \qquad 
\zeta(x,\omega):= \sum_{k=1}^\infty\sqrt{ \frac{\lambda_k}{\lambda_1} }\, \phi_k(x) \xi_k(\omega),
\]
Thus, setting $\varepsilon=O(\ell^{\frac{r}{2}})$ gives rise to $a(\ome,x)=\overline{a} + \varepsilon \zeta$,
which is the desired ``weak form" consisting of a deterministic field plus a small random perturbation. 
Therefore, our multi-modes Monte Carlo finite element method can still be applied to
such random elliptic PDEs in more general form.

It should be pointed out that the classical Karhunen-Lo\`eve expansion may be replaced by other
types of expansion formulas which may result in more efficient multi-modes Monte Carlo methods. 
The feasibility and competitiveness of non-Karhunen-Lo\`eve expansion technique will 
be investigated in forthcoming paper, where comparison among different expansion choices will 
also be studied. Finally, we also remark that the finite element method can be replaced by 
any other space discretization method such as finite difference, discontinuous Galerkin, and
spectral methods in the main algorithm.

\section*{Acknowledgments}
The research of the first author was partially supported by the NSF grant DMS-1318486 
and the research of the second author was supported by the NSF grant DMS-1417676. 


\end{document}